\documentclass[11pt,reqno]{amsart}
\usepackage[utf8]{inputenc}
\usepackage{amsthm,amsmath,amsxtra,latexsym,amscd,amssymb,exscale,mathrsfs,mathtools,physics,xparse,xcolor}
\usepackage[top=3cm, bottom=3.3cm, left=3cm, right=3cm]{geometry}
\usepackage{scalerel,geometry,boxedminipage,enumitem}
\usepackage[abbrev]{amsrefs}
\usepackage[colorlinks=true,linkcolor=blue]{hyperref}
\usepackage[depth=2]{bookmark}

\newtheorem{theorem}{Theorem}[section]
\newtheorem{lemma}[theorem]{Lemma}
\newtheorem{proposition}[theorem]{Proposition}
\newtheorem{corollary}[theorem]{Corollary}

\theoremstyle{definition}
\newtheorem{definition}[theorem]{Definition}
\theoremstyle{definition}
\newtheorem{remark}[theorem]{Remark}
\theoremstyle{definition}

\theoremstyle{definition}
\newtheorem{example}[theorem]{Example}
\theoremstyle{definition}
\newtheorem*{ex*}{Example}
\numberwithin{equation}{section}

\let\inter\undefined
\newcommand{\inter}{\textnormal{int}}
\let\Per\undefined
\newcommand{\Per}{\mathfrak{S}}
\let\Real\undefined
\newcommand{\Real}{\mathbb{R}}
\let\Complex\undefined
\newcommand{\Complex}{\mathbb{C}}
\let\one\undefined
\newcommand{\one}{\boldsymbol{1}}
\let\E\undefined
\newcommand{\E}{\mathscr{E}}
\let\P\undefined
\newcommand{\P}{\mathcal{P}}
\let\SS\undefined
\newcommand{\SS}{\mathcal{S}}

\title[Majorization for hyperbolic polynomials]{Determinant majorization for hyperbolic polynomials on Euclidean Jordan algebras}
\author{Doanh Pham}
\address{Beijing International Center for Mathematical Research, Peking University, Beijing, China}
\email{doanhpham@pku.edu.cn}

\calclayout

\begin{document}
\begin{abstract}
Under cone containment and the central ray hypothesis, we prove a determinant majorization result for hyperbolic polynomials on Euclidean Jordan algebras. In the case of real symmetric $n \times n$ matrices, this recovers the main theorem of Harvey and Lawson [Duke Math. J. 174 (2025), no. 13, 2749–2763] and also yields a new $\sigma_2$ majorization result. Moreover, for every $3 \leq k \leq n-1$, we construct explicit hyperbolic polynomials satisfying cone containment and the central ray hypothesis but for which the analogous $\sigma_k$ majorization fails.
\end{abstract}
\maketitle

\section{Introduction}	

Determinant majorization inequalities connect the theory of hyperbolic polynomials with fully nonlinear elliptic partial differential equations (PDEs).  In matrix form, such an inequality compares a G\r{a}rding--Dirichlet operator $F$ with the Monge-Amp\`ere operator by an estimate of the form
\[F(A)^{\frac{1}{N}} \geq F(I_n)^{\frac{1}{N}}(\det A)^{\frac{1}{n}}, \quad A>0,\]
where $N$ is the degree of $F$.  These inequalities have important applications in recent works on fully nonlinear PDEs; see, for example, \cite{Abja-Olive_AMPA2022, Abja-Dinew-Olive_PA2023, Guo-Phong_AOM2024_L_infinity_nonlinear, Guo-Phong_CAG2024_Entropy_energy_bounds, Guo-Phong-Tong_AOM2023_L_infinity_Monge-Ampere, Harvey-Lawson_CVPDE2023_det_major}.

Recall that a homogeneous polynomial $p$ on a real vector space $V$ is \textit{hyperbolic} with respect to a direction $a \in V$ if $p(a)>0$ and, for every $x \in V$, the one-variable polynomial $t \mapsto p(ta + x)$ has only real roots. The corresponding \textit{G\r{a}rding cone} $\Gamma(p)_+$ is the connected component of $\{p > 0\}$ containing $a$. By G\r{a}rding’s theorem (\cite{Garding_JMM1959_inequality_hyperbolic}), $p$ is hyperbolic with respect to every point of $\Gamma(p)_+$. A \textit{G\r{a}rding--Dirichlet operator} (\cite{Harvey-Lawson_CPAM2013_Garding}) is a hyperbolic polynomial on $\SS(n)$, the space of real symmetric $n \times n$ matrices, such that its G\r{a}rding cone contains the cone of positive definite matrices.

In \cite{Harvey-Lawson_CVPDE2023_det_major}, Harvey and Lawson proved determinant majorization for invariant G\r{a}rding-Dirichlet operators, including
operators invariant under the real orthogonal, complex unitary, and quaternionic unitary groups, and also treated the Lagrangian Monge-Amp\`ere operator. More recently in \cite{Harvey-Lawson_Duke2025_det_major}, they proved that determinant majorization still holds when the invariance assumption is relaxed to the central ray hypothesis: the gradient $\nabla F(I_n)$ is a positive multiple of $I_n$. Operators satisfying this hypothesis are called \textit{$I_n$-central}.

\begin{theorem}[\cite{Harvey-Lawson_Duke2025_det_major}] \label{thm: Harvey-Lawson det major}
For an $I_n$-central G\r{a}rding-Dirichlet operator $F$ of degree $N$ on $\SS(n)$, we have
\[F(A)^{\frac{1}{N}} \geq F(I_n)^{\frac{1}{N}} (\det A)^{\frac{1}{n}} \quad \text{for every} \;\; A > 0.\]
\end{theorem}

The cone of positive definite real symmetric matrices is the basic example of a symmetric cone. Euclidean Jordan algebras provide an intrinsic framework for symmetric cones, with the Jordan determinant and Jordan trace playing the roles of the usual determinant and trace. This makes it natural to ask whether determinant majorization is a genuinely symmetric cone phenomenon rather than a special feature of matrix algebras. Our main result is the following.

\begin{theorem} \label{thm: majorization in Jordan algebra}
Let $V$ be a finite-dimensional Euclidean Jordan algebra of rank $r$ with unit element $e$ and symmetric cone $\P_+$. Suppose that $p$ is a hyperbolic polynomial of degree $m$ on $V$ with G\r{a}rding cone $\Gamma_+$. If $\Gamma_+ \supset \P_+$ and $p$ is $e$-central, that is, $\tr^{p,e} = \frac{m}{r} \tr\nolimits_{\mathcal{J}}$, then
\[p(a)^{\frac{1}{m}} \geq p(e)^{\frac{1}{m}} \left(\det\nolimits_{\mathcal{J}}a\right)^{\frac{1}{r}} \quad \text{for every} \;\; a \in \P_+.\]
\end{theorem}

When $V = \SS(n)$ with the standard Jordan product $A\circ B=(AB+BA)/2$, the Jordan determinant is the usual determinant (see Example \ref{ex: Jordan algebra of symmetric matrices}). Thus, in this case, Theorem \ref{thm: majorization in Jordan algebra} recovers Theorem \ref{thm: Harvey-Lawson det major}.

The elementary symmetric function $\sigma_2$ on $\Real^n$, defined by
\[\sigma_2(x) = \sum_{i < j} x_i x_j,  \quad x = (x_1, \dots, x_n) \in \Real^n,\]
is a hyperbolic polynomial whose G\r{a}rding cone is denoted by $\Gamma(\sigma_2)_+$. The rank-two Lorentz Jordan algebra gives a symmetric cone realization of $\Gamma(\sigma_2)_+$. By applying Theorem \ref{thm: majorization in Jordan algebra} to this model, we prove a $\sigma_2$ majorization result in $\Real^n$. In the following, we denote $\one = (1, \dots, 1) \in \Real^n$.

\begin{corollary} \label{cor: sigma_2 majorization}
Suppose that $p$ is a $\one$-central hyperbolic polynomial of degree $m$ on $\Real^n$ with G\r{a}rding cone $\Gamma(p)_+$. If $\Gamma(p)_+ \supset \Gamma(\sigma_2)_+$, then
\[p(a)^{\frac{1}{m}} \geq p(\one)^{\frac{1}{m}} \left(\frac{\sigma_2(a)}{{n \choose 2}}\right)^{\frac{1}{2}} \quad \text{for every} \;\; a \in \Gamma(\sigma_2)_+.\]
\end{corollary}

For $A \in \SS(n)$, we define
\[\sigma_2(A) \coloneqq \sigma_2(\lambda_1(A), \dots, \lambda_n(A)),\]
where $\lambda_1(A), \dots, \lambda_n(A)$ are the usual eigenvalues of $A$. Then $\sigma_2$ is a hyperbolic polynomial on $\SS(n)$ and we denote its G\r{a}rding cone by $\Gamma(\sigma_2)_+$. Analogously to Corollary \ref{cor: sigma_2 majorization}, we also prove a $\sigma_2$ majorization result in $\SS(n)$.

\begin{corollary} \label{cor: sigma_2 majorization in matrices}
Suppose that $P$ is an $I_n$-central hyperbolic polynomial of degree $m$ on $\SS(n)$ with G\r{a}rding cone $\Gamma(P)_+$. If $\Gamma(P)_+ \supset \Gamma(\sigma_2)_+$, then
\[P(A)^{\frac{1}{m}} \geq P(I_n)^{\frac{1}{m}} \left(\frac{\sigma_2(A)}{{n \choose 2}}\right)^{\frac{1}{2}} \quad \text{for every} \;\; A \in \Gamma(\sigma_2)_+.\]
\end{corollary}

Although the determinant and $\sigma_2$ majorization results hold for hyperbolic polynomials on $\Real^n$ and $\SS(n)$ satisfying the central ray hypothesis and cone containment, in Section \ref{sec: On sigma_k}, we show that this phenomenon for $\sigma_2$ and $\sigma_n$ does not extend to the intermediate elementary symmetric functions. Specifically, for all $3 \leq k \leq n-1$, we construct explicit $\one$-central hyperbolic polynomials with the cone containment property but for which the analogous $\sigma_k$ majorization statement fails.

In Section \ref{sec: Dual}, we record equivalent gradient formulations of Theorem \ref{thm: majorization in Jordan algebra} and Corollaries \ref{cor: sigma_2 majorization} and \ref{cor: sigma_2 majorization in matrices} in terms of dual functions.

\section{Preliminaries on Jordan algebras}

In this section, we recall some basic facts about Jordan algebras. Further background can be found in \cite{Faraut_book_Analysis_symmetric_cones}.

\subsection{Basic definitions and the spectral theorem}

A real vector space $V$ is an \textit{algebra} if there is a bilinear product
\[V \times V \longmapsto V, \quad (x, y) \longmapsto x \circ y.\]
An algebra $V$ is called a \textit{Jordan algebra} if it satisfies
\[x \circ y = y \circ x \quad \text{and} \quad x \circ (x^2 \circ y) = x^2 \circ (x \circ y)\]
for all $x, y \in V$, where $x^2 \coloneqq x \circ x$. When $V$ is a Jordan algebra, we always assume that it has an element $e$, called the \textit{unit element}, which satisfies
\[e \circ x = x \quad \text{for all} \;\; x \in V.\]
We denote by $\Real[X]$ the algebra of one-variable real polynomials. The \textit{minimal polynomial} of $x \in V$ is the monic generator of the ideal $\{p \in \Real[X] : p(x) = 0\}$. Let $\deg(x)$ denote the degree of the minimal polynomial of $x$. The \textit{rank} of $V$ is defined by
\[\rank(V) \coloneqq \max_{x \in V} \deg(x).\]

A real finite-dimensional Jordan algebra $V$ is called \textit{Euclidean} if there exists an inner product $\langle \cdot, \cdot \rangle$ on $V$ which satisfies the associative property:
\[\langle x \circ y, z \rangle = \langle x, y \circ z \rangle \quad \text{for all} \;\; x, y, z \in V.\]
A vector $c \in V$ is called an \textit{idempotent} if
\[c^2 = c.\]
Two idempotents $c_1, c_2 \in V$ are called \textit{orthogonal} if they satisfy
\[c_1 \circ c_2 = 0.\]
A nonzero idempotent is called \textit{primitive} if it cannot be written as a sum of two nonzero orthogonal idempotents. A \textit{Jordan frame} is a set $\{c_1, \dots, c_k\}$ of pairwise-orthogonal primitive idempotents which satisfies
\[c_1 + \dots + c_k = e.\]

An important fact for Euclidean Jordan algebras is the spectral theorem stated as follows.

\begin{theorem}[{\cite[Theorem III.1.2]{Faraut_book_Analysis_symmetric_cones}}] \label{thm: spectral theorem for Jordan algebras}
Let $V$ be a finite-dimensional Euclidean Jordan algebra, and denote $r = \rank(V)$. Then for every $x \in V$, there exist a Jordan frame $\{c_1, \dots, c_r\}$ and real numbers $\lambda_1(x), \dots, \lambda_r(x)$ such that
\[x = \sum_{i=1}^r \lambda_i(x) c_i.\]
The numbers $\lambda_i(x)$'s, counted with multiplicities, are uniquely determined by $x$.
\end{theorem}

For each $x \in V$, the numbers $\lambda_i(x)$'s are called the \textit{Jordan eigenvalues} of $x$. We define
\[\tr\nolimits_{\mathcal{J}} x \coloneqq \sum_{i=1}^r \lambda_i(x) \quad \text{and} \quad \det\nolimits_{\mathcal{J}} x \coloneqq \prod_{i=1}^r \lambda_i(x).\]
By \cite[Proposition II.2.1 and Theorem III.1.2]{Faraut_book_Analysis_symmetric_cones}, $\tr\nolimits_{\mathcal{J}}$ and $\det\nolimits_{\mathcal{J}}$ are polynomials on $V$.

The \textit{cone of squares} is the closed convex cone defined by $\P \coloneqq \{x^2 : x \in V\}$. In spectral terms, we have
\[\P = \{x \in V : \, \lambda_i(x) \geq 0 \;\; \text{for all} \;\; i = 1, \dots, r\}.\]
The interior of $\P$, called the \textit{symmetric cone} of $V$, is the open convex cone
\[\P_+ \coloneqq \inter(\P) = \{x \in V : \, \lambda_i(x) > 0  \;\; \text{for all} \;\; i = 1, \dots, r\}.\]

\subsection{Basic models}
We are interested in the following models of Euclidean Jordan algebras.

\begin{example}[$\Real^n$]
The Euclidean space $\Real^n$ is a Euclidean Jordan algebra with the Jordan product given by
\[a \circ b = (a_1b_1, \dots, a_n b_n), \quad \text{for} \;\; a = (a_1, \dots, a_n) \;\; \text{and} \;\; b = (b_1, \dots, b_n).\]
The unit element is $\one \coloneqq (1, \dots, 1)$, and a Jordan frame is formed by the standard coordinate vectors. Thus for every $x = (x_1, \dots, x_n) \in \Real^n$, the Jordan eigenvalues of $x$ are just $x_1, \dots, x_n$. In particular,
\[ \tr\nolimits_{\mathcal{J}} x = x_1 + \dots + x_n \eqqcolon \sigma_1(x) \quad \text{and} \quad \det\nolimits_{\mathcal{J}} x = x_1 \dots x_n \eqqcolon \sigma_n(x).\]
In addition, the symmetric cone is $\P_+ = \Real^n_+$.
\end{example}

\begin{example}[$\SS(n)$] \label{ex: Jordan algebra of symmetric matrices}
Let $\SS(n)$ denote the space of $n \times n$ real symmetric matrices. Then $\SS(n)$ is a Euclidean Jordan algebra equipped with the Jordan product given by
\[A \circ B = \frac{AB + BA}{2} \quad \text{for} \;\; A, B \in \SS(n).\]
Clearly the unit element is $I_n$, the identity matrix. In addition, a Jordan frame on $\SS(n)$ is
\[c_i = u_i \otimes u_i, \quad i = 1, \dots, n,\]
where $u_1, \dots, u_n$ is an orthonormal basis of $\Real^n$. Thus, Jordan eigenvalues are just the usual matrix eigenvalues. In particular, $\tr\nolimits_{\mathcal{J}}$ and $\det\nolimits_{\mathcal{J}}$ coincide with the usual trace and determinant, respectively. The symmetric cone is $\P_+ = \P(n)_+$, the cone of positive definite symmetric $n \times n$ matrices.
\end{example}

The following special model of Euclidean Jordan algebras provides an important insight for the proof of Corollaries \ref{cor: sigma_2 majorization} and \ref{cor: sigma_2 majorization in matrices}.

\begin{example}[Lorentz] \label{ex: Lorentz cone}
Let $W$ be a finite-dimensional real vector space with an inner product $\langle \cdot, \cdot \rangle_W$. The vector space $V = \Real \times W$ equipped with the product
\[(t,u) \circ (s,v) = (ts + \langle u, v \rangle_W, tv + su)\]
is a Euclidean Jordan algebra with unit element $e = (1, 0)$. For every $x = (t, u) \in V$, set
\[\begin{cases*}
c_1 = \frac{1}{2}\left(1, \frac{u}{|u|}\right), \; c_2 = \frac{1}{2}\left(1, -\frac{u}{|u|}\right) & if  $u \neq 0$,\\
c_1 = \frac{1}{2}\left(1, \frac{v}{|v|}\right), \; c_2 = \frac{1}{2}\left(1, -\frac{v}{|v|}\right) & if $u = 0$,
\end{cases*}\]
where $v$ is an arbitrary nonzero vector in $W$. Then $(c_1, c_2)$ is a Jordan frame and $x = (t,u) = (t+|u|)c_1 + (t-|u|)c_2$. Thus the Jordan eigenvalues of $x = (t,u) \in V$ are
\[\lambda_1(x) = t - |u| \quad \text{and} \quad \lambda_2(x) = t + |u|.\]
Consequently, $\rank(V) = 2$, and for $x = (t, u) \in V$,
\begin{equation} \label{eq: Lorentz trace and det}
\tr\nolimits_{\mathcal{J}} x = 2t \quad \text{and} \quad \det\nolimits_{\mathcal{J}} x = t^2 - |u|^2.
\end{equation}
The symmetric cone of $V$ is
\[\P_+ = \{(t,u) \in \Real \times W : \, t > |u|\}.\]
\end{example}

\section{Hyperbolic polynomials and proofs of the majorization results}

\subsection{Hyperbolic polynomials and the central ray hypothesis}
We begin this section by providing some basic facts about hyperbolic polynomials. Further background can be found in \cite{Bauschke-Guler-Lewis-Sendov_CJM2001_Hyperbolic_poly, Harvey-Lawson_CPAM2013_Garding, Renegar_FoundCompMath2006_Hyperbolic_programs, Wagner_BAMS2011_Multivariate_stable_polynomials}. We assume that $V$ is a finite-dimensional real vector space. Let $p$ be a homogeneous polynomial of degree $m$ on $V$. For any $a, x \in V$ with $p(a) \neq 0$, the one-variable polynomial $s \mapsto p(sa + x)$ can be written as
\begin{equation}\label{eq: chap Garding theory, factor of homo poly}
p(sa + x) = p(a)\prod_{j=1}^m (s + \lambda^a_j(x)), \quad \lambda^a_j(x) \in \Complex \;\, \forall j.
\end{equation}
We define $\lambda^a(x) \coloneqq (\lambda^a_1(x), \dots, \lambda^a_m(x)) \in \Complex^m$ modulo the permutation group $\Per_m$.

A homogeneous polynomial $p$ of degree $m$ on $V$ is \textit{hyperbolic with respect to $a$} (or \textit{$a$-hyperbolic}), if $p(a) > 0$ and $\lambda^a(x) \in \Real^m$ for every $x \in V$. In this case, the \textit{G\r{a}rding cone} $\Gamma_+ \subset V$ (also denoted by $\Gamma_+^a$, or $\Gamma(p)_+$) is defined by
\[ \Gamma_+ = \{x \in V :\, \lambda^a_j(x) > 0 \;\; \text{for all} \;\; j = 1, \dots, m\}. \]
The closure of $\Gamma_+$ in $V$ is given by
\[\Gamma = \overline{\Gamma_+} = \{x \in V : \lambda^a_j(x) \geq 0 \;\; \text{for all} \;\; j = 1, \dots, m\}.\]
The fundamental result in the theory of hyperbolic polynomials is recorded in the following.

\begin{theorem}[\cite{Garding_JMM1959_inequality_hyperbolic}, see also \cite{Harvey-Lawson_CPAM2013_Garding, Renegar_FoundCompMath2006_Hyperbolic_programs}] \label{thm: Garding}
Suppose that $p$ is $a$-hyperbolic and $b \in \Gamma_+^a$. Then $p$ is also $b$-hyperbolic and $\Gamma_+^b = \Gamma_+^a$. Moreover, $\Gamma^a_+$ is the connected component of $\{p > 0\}$ containing $a$, and is convex.
\end{theorem}

For an $a$-hyperbolic polynomial $p$ of degree $m$ on $V$, we denote
\[\tr^{p,a}(x) \coloneqq \sum_{j=1}^{m} \lambda^{p,a}_j(x) \quad \text{for} \;\; x \in V.\]

\begin{definition}[Central ray hypothesis] \label{def: central ray}
Let $V$ be a finite-dimensional Euclidean Jordan algebra with the unit element $e$. An $e$-hyperbolic polynomial $p$ of degree $m$ on $V$ is said to be \textit{$e$-central} if there is a constant $\kappa > 0$ such that
\[\tr^{p,e}(x) = \kappa \tr\nolimits_{\mathcal{J}}(x) \quad \text{for all} \;\; x \in V.\]
By homogeneity, the constant $\kappa$ is necessarily $\kappa = \frac{m}{\rank(V)}$.
\end{definition}
Further useful properties of the central ray hypothesis for hyperbolic polynomials in an arbitrary vector space can be obtained by analogously following the arguments in \cite[Appendix E]{Harvey-Lawson_CVPDE2023_det_major} (where vector space is $\SS(n)$).

\begin{example} \label{ex: Jordan det is e-central hyperbolic}
Let $V$ be a finite-dimensional Euclidean Jordan algebra of rank $r$ with unit element $e$ and symmetric cone $\P_+$. The determinant function $\det\nolimits_{\mathcal{J}}$ on $V$ is an $e$-central hyperbolic polynomial of degree $r$ whose G\r{a}rding cone is $\Gamma_+ = \P_+$. Indeed, for each $x \in V$, let $c_1, \dots, c_r$ be a Jordan frame such that
\[x = \sum_{i=1}^r \lambda_i(x) c_i.\]
Then for every $t \in \Real$, we have
\[\det\nolimits_{\mathcal{J}}(te + x) = \det\nolimits_{\mathcal{J}}\left(\sum_{i=1}^r (t + \lambda_i(x))c_i\right) = \prod_{i=1}^r (t+\lambda_i(x)).\]
Thus $\det\nolimits_{\mathcal{J}}$ is $e$-hyperbolic and $\Gamma_+ = \P_+$. The $e$-central condition is obvious.
\end{example}

A hyperbolic polynomial in $\Real^n$ is said to be \textit{stable} if its G\r{a}rding cone contains $\Real^n_+$. Stable hyperbolic polynomials have only nonnegative coefficients (see \cite{Choe-Oxley-Sokal-Wagner_AdvApplMath2004_Homogeneous_multivariate_polynomials, Wagner_BAMS2011_Multivariate_stable_polynomials}). The combination of this fact and \cite[Basic Lemma 2.1]{Harvey-Lawson_CVPDE2023_det_major} (see also \cite[Fact 2.2]{Gurvits_EJC2008_hyperbolic_poly}) gives the following result which plays a crucial role in the proof of the main result. 

\begin{lemma} \label{lem: chap Garding theory, 1-central real stable homo poly majorizes det}
Suppose that $p$ is a $\one$-central stable hyperbolic polynomial of degree $m$ on $\Real^n$. Then for every $x = (x_1, \dots, x_n) \in \Real^n_+$, we have
\[p(x)^{\frac{1}{m}} \geq p(\one)^{\frac{1}{m}} (x_1 \dots x_n)^{\frac{1}{n}}.\]
\end{lemma}

Note that Lemma \ref{lem: chap Garding theory, 1-central real stable homo poly majorizes det} is the same as Theorem \ref{thm: majorization in Jordan algebra} when $V = \Real^n$.

\subsection{Proof of the majorization results}
In this subsection, we prove the main result and its consequences. The proof of Theorem \ref{thm: majorization in Jordan algebra} follows the same strategy as the proof of \cite[Main Theorem 1.3]{Harvey-Lawson_Duke2025_det_major}.

\begin{proof}[Proof of Theorem \ref{thm: majorization in Jordan algebra}]
We begin by fixing $a \in \P_+ \subset \Gamma_+$. By Theorem \ref{thm: spectral theorem for Jordan algebras}, there exists a Jordan frame $c_1, \dots, c_r \in V$ such that
\[a = \sum_{i=1}^r \alpha_i c_i,\]
where $\alpha_1, \dots, \alpha_r > 0$ are the Jordan eigenvalues of $a$. We consider a real polynomial in $\Real^r$ defined by
\[Q(x) \coloneqq p(x_1 c_1 + \dots + x_r c_r) \quad \text{for} \;\; x = (x_1, \dots, x_r) \in \Real^r.\]
For convenience, we define the linear map $L : \Real^r \to V$ by
\[L(x) = \sum_{i=1}^r x_i c_i \quad \text{for} \;\; x = (x_1, \dots, x_r) \in \Real^r.\]
In particular, $L(\one) = e$, the unit element of $V$, and $Q(x) = p(L(x))$. Important properties of $Q$ are proved in the following claims.

\textbf{Claim 1:} $Q$ is $m$-homogeneous and $Q(\one) = p(e)$.\\
\textit{Proof of Claim:} Since $p$ is $m$-homogeneous, for any $s > 0$ and $x \in \Real^r$, we have
\[Q(sx) = p(L(sx)) = p(sL(x)) = s^m p(L(x)) = s^m Q(x).\]
Moreover, $Q(\one) = p(L(\one)) = p(e)$.

\textbf{Claim 2:} $Q$ is stable.\\
\textit{Proof of Claim:} Let $v = (v_1, \dots, v_r) \in \Real^r_+$. Then since $(c_1, \dots, c_r)$ is a Jordan frame, the point $b =L(v)$ has eigenvalues $v_1, \dots, v_r > 0$, so $b \in \P_+$. Since $\P_+ \subset \Gamma_+$ by the assumption, $p$ is $b$-hyperbolic. Consequently, $Q(v) = p(b) > 0$ and for every $x \in \Real^r$, the one-variable polynomial
\[t \mapsto Q(tv + x) = p(tb + L(x))\]
has only real roots. Therefore $Q$ is $v$-hyperbolic. Since $v \in \Real^r_+$ is arbitrary and $\Real^r_+$ is connected, by Theorem \ref{thm: Garding}, we conclude that $Q$ is stable.

\textbf{Claim 3:} $Q$ is $\one$-central.\\
\textit{Proof of Claim:} Note that for every $x \in \Real^r$, the proof of Claim 2 yields
\[\lambda^{Q,\one}(x) = \lambda^{p,e}(L(x)).\]
Combining this and the assumption that $p$ is $e$-central, for every $x \in \Real^r$, we have
\begin{align*}
\tr^{Q, \one}(x) = \tr^{p,e}(L(x)) &= \frac{m}{r} \tr\nolimits_{\mathcal{J}}(L(x)) \\
&= \frac{m}{r}(x_1 + \dots + x_r).
\end{align*}
This proves that $Q$ is $\one$-central.

By these properties of $Q$, we apply Lemma \ref{lem: chap Garding theory, 1-central real stable homo poly majorizes det} to conclude that for every $x = (x_1, \dots, x_r) \in \Real^r_+$,
\[p(L(x))^\frac{1}{m} = Q(x)^{\frac{1}{m}} \geq Q(\one)^{\frac{1}{m}}(x_1 \dots x_r)^{\frac{1}{r}} = p(e)^{\frac{1}{m}}(x_1 \dots x_r)^{\frac{1}{r}}.\]
In particular, at $x = (\alpha_1, \dots, \alpha_r) \in \Real^r_+$, we get
\[p(a)^{\frac{1}{m}} \geq p(e)^{\frac{1}{m}}(\alpha_1 \dots \alpha_r)^{\frac{1}{r}} = p(e)^{\frac{1}{m}} \left(\det\nolimits_{\mathcal{J}}a\right)^{\frac{1}{r}}.\]
The proof is complete.
\end{proof}

We now prove Corollaries \ref{cor: sigma_2 majorization} and \ref{cor: sigma_2 majorization in matrices}. The idea is to apply Theorem \ref{thm: majorization in Jordan algebra} to the Lorentz model of Euclidean Jordan algebras (Example \ref{ex: Lorentz cone}).

\begin{proof}[Proof of Corollary \ref{cor: sigma_2 majorization}]
We denote $W = \one^\perp = \{y \in \Real^n : \langle y, \one \rangle = 0\}$. By the orthogonal decomposition $\Real^n = \Real\one \oplus W$, for every $x \in \Real^n$, we write
\[x = t\one + y \quad \text{where} \;\; t =\frac{\sigma_1(x)}{n} \;\; \text{and} \;\; y \in W.\]
Then for every $x \in \Real^n$, we have $|x|^2 = nt^2 + |y|^2$ and
\begin{align*}
\sigma_2(x) = \frac{\sigma_1(x)^2}{2} - \frac{|x|^2}{2} &= \frac{\sigma_1(x)^2}{2} - \frac{nt^2 + |y|^2}{2}\\
&= \frac{n(n-1)}{2} t^2 - \frac{|y|^2}{2}.
\end{align*}
Thus
\begin{equation}\label{eq: sigma_2 and det as Jordan algebra}
\frac{\sigma_2(x)}{{n \choose 2}} = t^2 - \frac{|y|^2}{n(n-1)} = t^2 - |u|^2,
\end{equation}
where we denote $u = \frac{y}{\sqrt{n(n-1)}} \in W$. Consequently,
\begin{align*}
\Gamma(\sigma_2)_+ &\equiv \{x : \, \sigma_1(x) > 0, \; \sigma_2(x) > 0\}\\
&= \{(t,u) : \, t > 0, \; t^2 - |u|^2 > 0\}\\
&=  \{(t,u) : \, t > |u|\}.
\end{align*}
Thus $\Gamma(\sigma_2)_+$ can be identified with the symmetric cone $\P_+$ of the Lorentz Jordan algebra $V = \Real\one \oplus W$ (see Example \ref{ex: Lorentz cone}) under the identification map $x \mapsto (t, u)$. In this way, the unit element $e = (1,0)$ of $V$ is identified with the vector $\one \in \Real^n$.

Now, let $p$ be a $\one$-central hyperbolic polynomial of degree $m$ on $\Real^n$ whose G\r{a}rding cone contains $\Gamma(\sigma_2)_+$. Then by the identification map $x \mapsto (t, u)$, $p$ is also an $e$-hyperbolic polynomial on $V$ whose G\r{a}rding cone (in $V$) contains $\P_+$. Moreover, since $p$ is $\one$-central in $\Real^n$, for every $x \in \Real^n$, we have
\[\tr^{p, e}(t, u) = \tr^{p, \one}(x) = \frac{m}{n} \sigma_1(x) = mt = \frac{m}{2} \tr\nolimits_{\mathcal{J}}(t,u),\]
where we have used \eqref{eq: Lorentz trace and det} to obtain the last equality. Thus $p$ is also $e$-central in $V$. Applying Theorem \ref{thm: majorization in Jordan algebra} for $p$ in $V$, we get
\[p(t,u)^{\frac{1}{m}} \geq p(e)^{\frac{1}{m}} \left(\det\nolimits_{\mathcal{J}}(t,u)\right)^{\frac{1}{2}} \quad \text{for every} \;\; (t,u) \in \P_+.\]
Recalling \eqref{eq: sigma_2 and det as Jordan algebra} and the fact that $\det\nolimits_{\mathcal{J}}(t,u) = t^2 - |u|^2$ by \eqref{eq: Lorentz trace and det}, we conclude that
\[p(x)^{\frac{1}{m}} \geq p(\one)^{\frac{1}{m}} \left(\frac{\sigma_2(x)}{{n \choose 2}}\right)^{\frac{1}{2}} \quad \text{for every} \;\; x \in \Gamma(\sigma_2)_+.\]
The proof is complete.
\end{proof}

\begin{proof}[Proof of Corollary \ref{cor: sigma_2 majorization in matrices}]
The proof is similar to that of Corollary \ref{cor: sigma_2 majorization}. We begin by denoting $W = \{Y \in \SS(n) : \sigma_1(Y) = 0\}$. By the orthogonal decomposition $\SS(n) = \Real I_n \oplus W$, for every $X \in \SS(n)$, we write
\[X = tI_n + Y \quad \text{where} \;\; t =\frac{\sigma_1(X)}{n} \;\; \text{and} \;\; Y \in W.\]
Then for every $X \in \SS(n)$, we have $|X|^2 = nt^2 + |Y|^2$ and
\begin{align*}
\sigma_2(X) = \frac{\sigma_1(X)^2}{2} - \frac{|X|^2}{2} &= \frac{\sigma_1(X)^2}{2} - \frac{nt^2 + |Y|^2}{2}\\
&= \frac{n(n-1)}{2} t^2 - \frac{|Y|^2}{2}.
\end{align*}
Thus
\begin{equation}\label{eq: sigma_2 and det as Jordan algebra matrices}
\frac{\sigma_2(X)}{{n \choose 2}} = t^2 - \frac{|Y|^2}{n(n-1)} = t^2 - |U|^2,
\end{equation}
where we denote $U = \frac{1}{\sqrt{n(n-1)}}Y \in W$. Consequently,
\begin{align*}
\Gamma(\sigma_2)_+ &\equiv \{X : \, \sigma_1(X) > 0, \; \sigma_2(X) > 0\}\\
&= \{(t,U) : \, t > 0, \; t^2 - |U|^2 > 0\}\\
&=  \{(t,U) : \, t > |U|\}.
\end{align*}
Thus $\Gamma(\sigma_2)_+$ can be identified with the symmetric cone $\P_+$ of the Lorentz Jordan algebra $V = \Real I_n \oplus W$ (see Example \ref{ex: Lorentz cone}) under the identification map $X \mapsto (t, U)$. In this way, the unit element $e = (1,0)$ of $V$ is identified with the identity matrix $I_n$.

Now, let $P$ be a $I_n$-central hyperbolic polynomial of degree $m$ on $\SS(n)$ whose G\r{a}rding cone contains $\Gamma(\sigma_2)_+$. Then by the identification map $X \mapsto (t, U)$, $P$ is also an $e$-hyperbolic polynomial on $V$ whose G\r{a}rding cone (in $V$) contains $\P_+$. Moreover, since $P$ is $I_n$-central in $\SS(n)$, for every $X \in \SS(n)$, we have
\[\tr^{P, e}(t, U) = \tr^{P, I_n}(X) = \frac{m}{n} \sigma_1(X) = mt = \frac{m}{2} \tr\nolimits_{\mathcal{J}}(t,U),\]
where we have used \eqref{eq: Lorentz trace and det} to obtain the last equality. Thus $P$ is also $e$-central in $V$. Applying Theorem \ref{thm: majorization in Jordan algebra} for $P$ in $V$, we get
\[P(t,U)^{\frac{1}{m}} \geq P(e)^{\frac{1}{m}} \left(\det\nolimits_{\mathcal{J}}(t,U)\right)^{\frac{1}{2}} \quad \text{for every} \;\; (t,U) \in \P_+.\]
Recalling \eqref{eq: sigma_2 and det as Jordan algebra matrices} and the fact that $\det\nolimits_{\mathcal{J}}(t,U) = t^2 - |U|^2$ by \eqref{eq: Lorentz trace and det}, we conclude that
\[P(X)^{\frac{1}{m}} \geq P(I_n)^{\frac{1}{m}} \left(\frac{\sigma_2(X)}{{n \choose 2}}\right)^{\frac{1}{2}} \quad \text{for every} \;\; X \in \Gamma(\sigma_2)_+.\]
The proof is complete.
\end{proof}

\section{On $\sigma_k$ majorization} \label{sec: On sigma_k}

For each $1 \leq k \leq n$, the elementary symmetric function $\sigma_k$ on $\Real^n$, defined by
\[\sigma_k(x) = \sum_{i_1 < \dots < i_k} x_{i_1} \dots x_{i_k} \quad \text{for} \;\; x = (x_1, \dots, x_n) \in \Real^n,\]
is a $\one$-central hyperbolic polynomial. We denote its G\r{a}rding cone by $\Gamma(n,k)_+$.

For each $1 \leq k \leq n$ and $X \in \SS(n)$, we define
\[\sigma_k(X) \coloneqq \sigma_k(\lambda(X)),\]
where $\lambda(X) \coloneqq (\lambda_1(X), \dots, \lambda_n(X))$ (defined modulo the permutation group $\mathfrak{S}_n$) is the vector whose coordinates are the usual matrix eigenvalues of $X$. Then $\sigma_k$ is an $I_n$-central hyperbolic polynomial on $\SS(n)$. We denote its G\r{a}rding cone by $\Gamma_{\SS}(n,k)_+$. Then
\[\Gamma_{\SS}(n,k)_+ = \{X \in \SS(n) : \lambda(X) \in \Gamma(n,k)_+\}.\]

The main point of this section is that for $3 \leq k \leq n-1$, the central ray hypothesis together with cone containment is not sufficient for hyperbolic $\sigma_k$ majorization, either on $\Real^n$ or on $\SS(n)$, in contrast with Corollary \ref{cor: sigma_2 majorization} and Lemma \ref{lem: chap Garding theory, 1-central real stable homo poly majorizes det}. After recalling two useful facts about elementary symmetric functions, we construct $\one$-central ($I_n$-central) hyperbolic polynomials that violate $\sigma_k$ majorization for every $3 \leq k \leq n-1$.

\begin{lemma}[see e.g. \cite{Lin-Trudinger_BAMS1994_Inequalities_elem_sym_funct}] \label{lem: sum of n-k+1 coordinates is positive}
Let $2 \leq k \leq n$. If $x = (x_1, \dots, x_n) \in \Gamma(n,k)_+$, then
\[(x_1, \dots, x_{n-\ell}) \in \Gamma(n-\ell, k-\ell)_+ \quad \text{for all} \;\; 1 \leq \ell \leq k-1. \]
In particular, if $x = (x_1, \dots, x_n) \in \Gamma(n,k)_+$, then
\[x_1 + \dots + x_{n-k+1} > 0.\]
\end{lemma}

Before stating the next lemma, we define, for $1 \leq k \leq n$, the linear function $L_k$ on $\Real^n$ by
\[L_k(x) = \frac{x_1 + \dots + x_{n-k+1}}{n-k+1} \quad \text{for} \;\; x \in \Real^n.\]

\begin{lemma} \label{lem: Trudinger's inequality}
For each $2 \leq k \leq n$, we have
\[\frac{\sigma_1(x)}{n} \geq \frac{n-k+1}{2n-k} L_k(x) \quad \text{for all} \;\; x \in \Gamma(n,k)_+.\]
\end{lemma}

\begin{proof}
Let $x = (x_1, \dots, x_n) \in  \Gamma(n,k)_+$. By \cite[Eq. (1.16)]{Kuo-Trudinger_IUMJ2007_New_max_principle}, we have
\[x_i \geq - \frac{n-k}{n(k-1)} \sigma_1(x) \quad \text{for each} \;\; i = 1, \dots, n.\]
Summing this over the last $(k-1)$ coordinates yields
\[\sum_{i=n-k+2}^{n} x_i \geq - \frac{n-k}{n} \sigma_1(x).\]
Therefore,
\[(n-k+1)L_k(x) = \sigma_1(x) - \sum_{i=n-k+2}^{n} x_i \leq \frac{2n-k}{n} \sigma_1(x).\]
\end{proof}

We now construct $\one$-central ($I_n$-central) hyperbolic polynomials violating $\sigma_k$ majorization for all $n \geq 4$ and $3 \leq k \leq n-1$.

\begin{example}[$\Real^n$] \label{ex: hyperbolic poly violate sigma_k in R^n}
For $2 \leq k \leq n$, note that
\[\frac{n-k+1}{2n-k} > \frac{1}{k} \quad \Longleftrightarrow \quad 2 < k < n.\]
Let $3 \leq k \leq n-1$. We fix a rational number $q = \frac{r}{m}$, $r, m \in \mathbb{N}$, such that
\[\frac{1}{k} < q < \frac{n-k+1}{2n-k}.\]
Define the linear function $L$ on $\Real^n$ by
\[L(x) = \frac{1}{1-q} \left(\frac{\sigma_1(x)}{n} - q L_k(x)\right) \quad \text{for} \;\; x \in \Real^n.\]
Then $L(x) > 0$ for all $x \in \Gamma(n,k)_+$ by Lemma \ref{lem: Trudinger's inequality}, and $L(\one) = 1$. We consider the polynomial $p$ on $\Real^n$ given by
\[p(x) = L_k(x)^r L(x)^{m-r}  \quad \text{for} \;\; x \in \Real^n.\]
Then $p$ is $m$-homogeneous, $p(\one) = 1$, and for all $t \in \Real$, $x \in \Real^n$, 
\[p(t\one + x) = (t + L_k(x))^r (t + L(x))^{m-r}.\]
Thus $p$ is $\one$-hyperbolic with G\r{a}rding cone
\[\Gamma(p)_+ = \{x \in \Real^n : L_k(x) > 0, \; L(x) > 0\}.\]
Consequently, $\Gamma(p)_+ \supset \Gamma(n,k)_+$ since $L(x) > 0$ and $L_k(x) > 0$ (by Lemma \ref{lem: sum of n-k+1 coordinates is positive}) for all $x \in \Gamma(n,k)_+$. Moreover, $p$ is $\one$-central since
\[\tr^{p, \one}(x) = r L_k(x) + (m-r) L(x) = m \left(qL_k(x) + (1-q)L(x) \right) = \frac{m \sigma_1(x)}{n}.\]

We put $\alpha = n-k+1$, $\beta = k-1$, and consider the point
\[x_{\varepsilon} = (\underbrace{\varepsilon, \dots, \varepsilon}_{\alpha}, \underbrace{1, \dots, 1}_{\beta}), \quad \varepsilon > 0.\]
Then $x_{\varepsilon} \in \Real^n_+ \subset \Gamma(n,k)_+$. Moreover, $L_k(x_\varepsilon) = \varepsilon$, and
\[L(x_{\varepsilon}) = \frac{\frac{\alpha\varepsilon + \beta}{n} - q\varepsilon}{1-q} \to \frac{\beta}{n(1-q)} \quad \text{as} \;\; \varepsilon \to 0.\]
Consequently, $p(x_{\varepsilon})^{\frac{1}{m}} \sim C \varepsilon^q.$ Meanwhile, $\sigma_k(x_{\varepsilon}) = \alpha \varepsilon + O(\varepsilon^2)$, so $\sigma_k(x_{\varepsilon})^{\frac{1}{k}} \sim C'\varepsilon^{\frac{1}{k}}$. Since $q > \frac{1}{k}$, this implies that for sufficiently small $\varepsilon > 0$, we have
\[p(x_{\varepsilon})^{\frac{1}{m}} < p(\one)^{\frac{1}{m}} \left(\frac{\sigma_k(x_\varepsilon)}{{n \choose k}}\right)^{\frac{1}{k}}.\]
So $\sigma_k$ majorization fails for $p$.
\end{example}

\begin{example}[$\SS(n)$] \label{ex: hyperbolic poly violate sigma_k in matrices}
Let $3 \leq k \leq n-1$ and recall the hyperbolic polynomial $p$ of degree $m$ on $\Real^n$ constructed in Example \ref{ex: hyperbolic poly violate sigma_k in R^n}. For each $X = (X_{ij}) \in \SS(n)$, we denote
\[d(X) \coloneqq (X_{11}, \dots, X_{nn}).\]
Consider the polynomial $F$ on $\SS(n)$ defined by
\[F(X) = p(d(X)) \quad \text{for} \;\; X \in \SS(n).\]
Then for every $t \in \Real$ and $X \in \SS(n)$, we have
\[F(tI_n + X) = p(t\one + d(X)).\]
Thus $F$ is $I_n$-hyperbolic with
\[\lambda^{F, I_n}(X) = \lambda^{p, \one}(d(X)).\]
It follows that
\[\tr^{F, I_n}(X) = \tr^{p, \one}(d(X)) = \frac{m\sigma_1(d(X))}{n} = \frac{m \tr(X)}{n},\]
which means that $F$ is $I_n$-central. In addition,
\begin{equation} \label{eq: example of matrix nonmajor, Garding cone}
\Gamma(F)_+ = \{X \in \SS(n) : d(X) \in \Gamma(p)_+\}.
\end{equation}
To see cone containment, let $X \in \Gamma_{\SS}(n,k)_+$. Then $\lambda(X) \in \Gamma(n,k)_+$. Moreover, by Schur-Horn theorem and Birkhoff’s theorem (see \cite[Theorems 4.3.45 and 4.3.49]{Horn-Johnson_Book_Matrix}), we have
\[d(X) \in \mathrm{conv}\{\pi(\lambda(X)) : \pi \in \mathfrak{S}_n\}.\]
Since $\Gamma(n,k)_+$ is convex and invariant under coordinate permutations, we conclude that $d(X) \in \Gamma(n,k)_+ \subset \Gamma(p)_+$. By \eqref{eq: example of matrix nonmajor, Garding cone}, this implies $X \in \Gamma(F)_+$. Hence $\Gamma(F)_+ \supset \Gamma_{\SS}(n,k)_+$.

Now, we consider the positive definite diagonal matrix
\[X_{\varepsilon} \coloneqq \mathrm{diag}(\underbrace{\varepsilon, \dots, \varepsilon}_{n-k+1}, \underbrace{1, \dots, 1}_{k-1}), \quad \varepsilon > 0.\]
Following the argument in Example \ref{ex: hyperbolic poly violate sigma_k in R^n}, for sufficiently small $\varepsilon > 0$, we have
\[F(X_{\varepsilon})^{\frac{1}{m}} < F(I_n)^{\frac{1}{m}} \left(\frac{\sigma_k(X_\varepsilon)}{{n \choose k}}\right)^{\frac{1}{k}}.\]
So $\sigma_k$ majorization fails for $F$.
\end{example}

We note, however, that the polynomial constructed in Example \ref{ex: hyperbolic poly violate sigma_k in R^n} (resp. Example \ref{ex: hyperbolic poly violate sigma_k in matrices}) is not invariant under the permutation group $\mathfrak{S}_n$ (resp. the real orthogonal group $\mathrm{O}(n)$). Thus, if one additionally assumes the invariance condition, either on $\Real^n$ or on $\SS(n)$, it remains inconclusive whether the analogous $\sigma_k$ majorization holds for $3 \leq k \leq n-1$.

\section{Dual functions} \label{sec: Dual}

In this section, we always assume that $p$ is a hyperbolic polynomial of degree $m$ on a finite-dimensional real vector space $V$. The G\r{a}rding cone of $p$ is denoted by $\Gamma_+$ and its closure in $V$ is denoted by $\Gamma$. Define
\[f(x) = p(x)^{\frac{1}{m}}, \quad x \in \Gamma.\]
Then $f$ is $1$-homogeneous on $\Gamma$. Moreover, $f$ and $\log f$ are concave in $\Gamma_+$. Specifically, by \cite[Proposition D.3]{Harvey-Lawson_CVPDE2023_det_major}, we have
\[[D^2 f(a)](v,v) = -\frac{f(a)}{m^2} \Big(m|\lambda^a(v)|^2 - \sigma_1(\lambda^a(v))^2\Big) \quad \text{for all} \;\; a \in \Gamma_+ \;\; \text{and} \;\; v \in V,\]
and
\begin{equation} \label{eq: Hessian of log hyperbolic}
[D^2 \log f(a)](v, v) = - \frac{|\lambda^a(v)|^2}{m} \quad \text{for all} \;\; a \in \Gamma_+ \;\; \text{and} \;\; v \in V.
\end{equation}

The edge $\E$ of the G\r{a}rding cone $\Gamma_+$ (or $\Gamma$) is defined by
\[\E \coloneqq \Gamma \cap (-\Gamma).\]
The edge $\E$ is also the set of all $x \in V$ such that $\lambda^a(x) = 0$ for some $a \in \Gamma_+$, and hence for all $a \in \Gamma_+$ (see e.g. \cite{Harvey-Lawson_CPAM2013_Garding}).

The G\r{a}rding cone $\Gamma_+$ (or $\Gamma$) is said to be \textit{regular}\footnote{This is called \textit{complete} in \cite{Harvey-Lawson_CPAM2013_Garding}.} if its edge $\E = \{0\}$. If $\Gamma$ is regular, then the dual cone of $\Gamma$,
\[\Gamma^* \coloneqq \{u \in V^* :\, \langle u, x \rangle \geq 0 \;\; \text{for all} \;\; x \in \Gamma\},\]
is full-dimensional, and whose interior, $\Gamma^*_+$, satisfies $\Gamma^* = \overline{\Gamma^*_+}$. Moreover, the bipolar theorem gives $(\Gamma^*)^* = \Gamma$. Here, we identify $(V^*)^*$ with $V$.

\begin{proposition} \label{prop: gradient of log regular hyperbolic poly is real analytic diffeo}
Suppose that $p$ is a hyperbolic polynomial on $V$ with G\r{a}rding cone $\Gamma_+$. If $\Gamma$ is regular, then $\nabla \log f : \Gamma_+ \to \Gamma^*_+$ is a real-analytic diffeomorphism.
\end{proposition}

\begin{proof}
Assume $\Gamma$ is regular. By \cite[Theorem 27]{Renegar_FoundCompMath2006_Hyperbolic_programs} and its proof, $\nabla \log f : \Gamma_+ \to \Gamma^*_+$ is a bijection. In addition, $\nabla \log f$ is real-analytic in $\Gamma_+$ since $p$ is a polynomial which is positive in $\Gamma_+$. Thus, by \eqref{eq: Hessian of log hyperbolic} and the inverse function theorem, $\nabla \log f : \Gamma_+ \to \Gamma^*_+$ is a real-analytic diffeomorphism.
\end{proof}

We recall the following elementary inequality which is of great importance for the theory of hyperbolic polynomials. A proof (taken from \cite[Lemma 5.3]{Harvey-Lawson_CPAM2013_Garding}) is provided here for the sake of exposition.

\begin{lemma} \label{lem: chap Garding theory, Garding AM-GM}
For any $a, x \in \Gamma_+$, we have
\[ \frac{f(x)}{f(a)} \leq \langle \nabla \log f(a), x \rangle, \]
with equality if and only if $\lambda^a_1(x) = \dots = \lambda^a_m(x)$.
\end{lemma}

\begin{proof}
Let $a \in \Gamma_+$. We recall \eqref{eq: chap Garding theory, factor of homo poly} that
\begin{equation} \label{eq: chap Garding theory, factor of homo poly again}
p(sa + x) = p(a)\prod_{j=1}^m (s + \lambda^a_j(x)) \quad \text{for every} \;\; s \in \Real \;\; \text{and} \;\; x \in V.
\end{equation}
In particular, taking $s = 0$ in \eqref{eq: chap Garding theory, factor of homo poly again}, we have
\begin{equation} \label{eq: p(x) = p(a) product eigenvalues}
p(x) = p(a) \prod_{j=1}^m \lambda^a_j(x) \quad \text{for every} \;\; x \in V. 
\end{equation}
Furthermore, since $p$ is $m$-homogeneous, one can see from \eqref{eq: chap Garding theory, factor of homo poly again} that
\[\lambda^a(tx) = t\lambda^a(x) \quad \text{for every} \;\; t \in \Real \;\; \text{and} \;\; x \in V.\]
Thus, taking $s = 1$ and replacing $x$ by $tx$ in \eqref{eq: chap Garding theory, factor of homo poly again}, we get
\begin{equation} \label{eq: factor p(a +tx)}
p(a + tx) = p(a) \prod_{j=1}^m (1 + t \lambda^a_j(x)) \quad \text{for all} \;\, t \in \Real \text{ and } x \in V.
\end{equation}
Then, taking the logarithmic derivative of \eqref{eq: factor p(a +tx)} at $t=0$ yields
\begin{equation} \label{eq: chap Garding theory, trace as gradient of log}	\langle \nabla \log p(a), x\rangle = \sum_{j=1}^{m} \lambda^a_j(x) \quad \text{for every} \;\; x \in V.
\end{equation}
If $x \in \Gamma_+$, then $\lambda^a_j(x) > 0$ for each $j = 1, \dots, m$. Therefore, since $f(x) = p(x)^{\frac{1}{m}}$, the assertion follows from \eqref{eq: p(x) = p(a) product eigenvalues}, \eqref{eq: chap Garding theory, trace as gradient of log} and the AM-GM inequality.
\end{proof}

The dual function $f^*$ of $f$ is defined by
\begin{equation} \label{eq: dual function}
f^*(u) \coloneqq \inf_{x \in \Gamma_+} \frac{\langle u, x \rangle}{f(x)} = \inf_{x \in \Gamma_+, \, f(x) = 1} \langle u,x \rangle, \quad \text{for} \;\; u \in \Gamma^*.
\end{equation}
In the context of hyperbolic polynomials, $f^*$ has several useful properties listed in the following.

\begin{theorem} \label{prop: dual function}
Assume that the G\r{a}rding cone $\Gamma_+$ is regular. The following properties of $f^*$ hold.
\begin{enumerate}[leftmargin=\parindent+0.5cm,label=\textnormal{(\alph*)}]
\item If $u \in \Gamma^*_+$, then $f^*(u) > 0$ and the infimum in \eqref{eq: dual function} is attained:
\[f^*(u) = \min_{x \in \Gamma_+, \, f(x) = 1} \langle u, x \rangle = \frac{1}{f(a)} = \frac{1}{f((\nabla \log f)^{-1}(u))},\]
where $a \in \Gamma_+$ satisfies $\nabla \log f(a) = u$, and the minimum is attained only at $x = \frac{a}{f(a)}$.
\item $f^*$ is real-analytic in $\Gamma^*_+$.
\item $f^*(\nabla f(a)) = 1$ for every $a \in \Gamma_+$.
\item If $u \in \partial \Gamma^* \backslash \{0\}$, then $f^*(u) = 0$, so the infimum in \eqref{eq: dual function} is not attained at any point in $\Gamma_+$.
\item $f^*$ is concave, continuous, and $1$-homogeneous on $\Gamma^*$.
\item The basic inequality is
\[\langle u, a \rangle \geq f^*(u) f(a) \quad \text{for all} \;\; u \in \Gamma^*_+ \;\; \text{and} \;\; a \in \Gamma_+,\]
with equality if and only if $\nabla f(a) = \frac{u}{f^* (u)}$ or, equivalently, $\nabla f^*(u) = \frac{a}{f(a)}$.
\item The gradients are related by
\[\nabla f^*(\nabla f(a)) = \frac{a}{f(a)}, \quad \nabla f(\nabla f^*(u)) = \frac{u}{f^*(u)} \quad \; \text{for all} \;\; u \in \Gamma^*_+ \;\; \text{and} \;\; a \in \Gamma_+.\]
In particular, $(f \nabla f) = (f^* \nabla f^*)^{-1}$ in $\Gamma_+$.
\item $(f^*)^* = f$.
\end{enumerate}
\end{theorem}

\begin{proof}
We prove these statements one by one.

\textit{Proof of (a):} Let $u \in \Gamma^*_+$. Since $\Gamma$ is regular, by Proposition \ref{prop: gradient of log regular hyperbolic poly is real analytic diffeo}, there exists a unique point $a \in \Gamma_+$ such that $\nabla \log f(a) = u$. By Lemma \ref{lem: chap Garding theory, Garding AM-GM}, it follows that
\[\frac{\langle u, x \rangle}{f(x)} = \frac{\langle \nabla \log f(a), x \rangle}{f(x)} \geq \frac{1}{f(a)} \quad \text{for every} \;\; x \in \Gamma_+,\]
with equality if and only if $\lambda_1^a(x) = \dots = \lambda_m^a(x)$. Clearly, $x = a/f(a)$ satisfies the equality condition. In addition, suppose $y \in \Gamma_+$ satisfies $\lambda^a(y) = c\one$ for some $c > 0$. Then by the fact that
\[\lambda^a(ta + x) = t\one + \lambda^a(x) \quad \text{for all} \;\; t \in \Real \;\; \text{and} \;\; x \in V,\]
which can be obtained from \eqref{eq: chap Garding theory, factor of homo poly again}, we deduce that $\lambda^a(y - ca) = 0$. This means $y-ca \in \E$. Since $\Gamma$ is regular, we must have $y = ca$.\\

\textit{Proof of (b):} This follows from (a) and Proposition \ref{prop: gradient of log regular hyperbolic poly is real analytic diffeo}.\\

\textit{Proof of (c):} Note that since $f$ is concave and $1$-homogeneous, we have
\[\frac{\langle \nabla f(a), x \rangle}{f(x)} \geq 1 \quad \text{for every} \;\; x \in \Gamma_+,\]
where the equality holds at $x = a$.\\

\textit{Proof of (d):} Let $u \in \partial \Gamma^* \backslash \{0\}$. Then $\langle u, x \rangle \geq 0$ for all $x \in \Gamma$, so $f^*(u) \geq 0$. In addition, since $u \in \partial \Gamma^*$, by an interior cone condition (see e.g. \cite[Proposition I.1.4]{Faraut_book_Analysis_symmetric_cones}), there exists $z \in \Gamma\backslash\{0\}$ such that $\langle u, z \rangle = 0$. If $z \in \Gamma_+$, then there exists $\varepsilon > 0$ such that $z + \varepsilon u \in \Gamma_+$ and $z - \varepsilon u \in \Gamma_+$. It follows that
\[\varepsilon |u|^2 = \langle u, z + \varepsilon u \rangle \geq 0 \quad \text{and} \quad -\varepsilon |u|^2 = \langle u, z - \varepsilon u \rangle \geq 0.\]
This contradiction implies that $z \in \partial \Gamma$, so $p(z) = 0$. Let $a \in \Gamma_+$ and let $r$ be the rank of $z$, i.e.,
\[r = \#\{j \in [m] : \lambda_j^a(z) \neq 0\}.\]
Then $r \geq 1$ since $\Gamma$ is regular and $z \neq 0$. Thus by \eqref{eq: factor p(a +tx)}, we have
\[p(a + tz) = p(a)\prod_{j=1}^m (1 + t \lambda^a_j(z)) \geq Ct^r \quad \text{for all} \;\; t \geq 0,\]
for some constant $C > 0$. Since $a + tz \in \Gamma_+$ for all $t \geq 0$, it follows that
\[f^*(u) \leq \frac{\langle u, a + tz \rangle}{f(a + tz)} = \frac{\langle u,a \rangle}{f(a + tz)} \to 0 \quad \text{as} \;\; t \to \infty.\]
Hence $f^*(u) = 0$, and the infimum in \eqref{eq: dual function} is not attained at any point in $\Gamma_+$.\\

\textit{Proof of (e):} As the pointwise infimum of a family of linear functions, $f^*$ is $1$-homogeneous and concave in $\Gamma^*$. In addition, it is clear that $f^*$ is continuous at $0$. It remains to prove the continuity of $f^*$ at $\partial \Gamma^* \backslash \{0\}$.

Let $a \in \Gamma_+$ and $u \in \partial \Gamma^* \backslash \{0\}$. The argument in the proof of (d) shows that there exists $z \in \partial\Gamma \backslash \{0\}$ satisfying $\langle u, z \rangle = 0$. Let $r$ be the rank of $z$. Then since $p(z) = 0$ and $\Gamma$ is regular, we have $1 \leq r \leq m-1$. Moreover, by \eqref{eq: factor p(a +tx)}, we have
\[p(a + tz) = p(a)\prod_{j=1}^m (1 + t \lambda^a_j(z)) \geq Ct^r \quad \text{for all} \;\; t \geq 0,\]
for some constant $C > 0$. Hence,
\begin{equation} \label{eq: estimate f(a+tz) for z on boundary}
f(a+tz) = p(a+tz)^{\frac{1}{m}} \geq Ct^{\frac{r}{m}} \quad \text{for all} \;\; t \geq 0,
\end{equation}
for some constant $C > 0$.

Suppose that $\{u_k\}_{k \in \mathbb N}$ is a sequence in $\Gamma^*$ such that $u_k \to u$ as $k \to \infty$.  Since $\eval{f^*}_{\partial\Gamma^*} = 0$ by (d), we need to prove $f^*(u_k) \to 0$. It suffices to assume $u_k \in \Gamma^*_+$ for all $k \in \mathbb N$. Then setting $\delta_k = \langle u_k, z \rangle$, by an interior cone condition (see e.g. \cite[Proposition I.1.4]{Faraut_book_Analysis_symmetric_cones}), we have $\delta_k > 0$ and
\[\delta_k \to \langle u, z \rangle = 0 \quad \text{as} \;\; k \to \infty. \]
Furthermore, by \eqref{eq: estimate f(a+tz) for z on boundary}, for every $t \geq 0$, since $a + tz \in \Gamma_+$, we get
\begin{equation} \label{eq: decay of f(uk)}
f^*(u_k) \leq \frac{\langle u_k, a + tz \rangle}{f(a + tz)} = \frac{\langle u_k, a \rangle + t\delta_k}{f(a + tz)} \leq Ct^{-\frac{r}{m}} + C \delta_k t^{1-\frac{r}{m}},
\end{equation}
for some constant $C > 0$. Choosing any $0 < \alpha < \frac{m}{m-r}$ and then setting $t = \delta_k^{-\alpha}$ in \eqref{eq: decay of f(uk)}, we obtain
\[f^*(u_k) \leq C \delta_k^{\frac{\alpha r}{m}} + C \delta_k^{1 - \alpha(1 - \frac{r}{m})} \to 0 \quad \text{as} \;\; k \to \infty.\]
This completes the proof of (e). \\

\textit{Proof of (f):} The inequality of (f) simply follows from the definition \eqref{eq: dual function}. Moreover, if the equality holds, i.e., $\langle u, a \rangle = f^*(u) f(a)$ for some $u \in \Gamma^*_+$ and $a \in \Gamma_+$, then
\[f(a) = \left\langle \frac{u}{f^*(u)}, a \right\rangle \quad \text{while} \quad f(x) \leq \left\langle  \frac{u}{f^*(u)}, x \right\rangle \;\; \text{for all} \;\; x \in \Gamma_+.\]
Since $f$ is concave and $1$-homogeneous, this implies that $\nabla f(a) = u/f^*(u)$. By (b) and (e), this argument also shows that $\nabla f^*(u) = a/f(a).$\\

\textit{Proof of (g):} Note that for any $a \in \Gamma_+$ and $u \in \Gamma^*_+$, by (f) we have
\[\nabla f(a) = \frac{u}{f^*(u)} \quad \Longleftrightarrow \quad \nabla f^*(u) = \frac{a}{f(a)}.\]
We combine this and the fact that $\nabla f$ and $\nabla f^*$ are $0$-homogeneous to conclude (g).\\

\textit{Proof of (h):} Let $x \in \Gamma_+$. The basic inequality in (f) gives
\[f(x) \leq \frac{\langle u, x \rangle}{f^*(u)} \quad \text{for every} \;\; u \in \Gamma^*_+.\]
Taking the infimum over $u \in \Gamma^*_+$ of the right-hand side yields $f(x) \leq (f^*)^*(x)$. To see the reverse inequality, take $u = \nabla f(x)$. Then $u \in \Gamma^*_+$ by Proposition \ref{prop: gradient of log regular hyperbolic poly is real analytic diffeo} and
\[(f^*)^*(x) \leq \frac{\langle u, x \rangle}{f^*(u)} =  \frac{\langle \nabla f(x), x \rangle}{f^*(\nabla f(x))} = f(x),\]
where we have used (c) to obtain the last equality. Therefore $(f^*)^*(x) = f(x)$.

For $x \in \partial\Gamma$, we let $a \in \Gamma_+$. For any $\varepsilon > 0$, since $x + \varepsilon a \in \Gamma_+$, we have \[(f^*)^*(x + \varepsilon a) = f(x + \varepsilon a).\]
Since $\langle a, u \rangle \geq 0$ for all $u \in \Gamma^*_+$, this implies
\[0 \leq (f^*)^*(x) = \inf_{u \in \Gamma^*_+} \frac{\langle u, x \rangle}{f^*(u)} \leq \inf_{u \in \Gamma^*_+} \frac{\langle u, x + \varepsilon a \rangle}{f^*(u)} = (f^*)^*(x + \varepsilon a) = f(x + \varepsilon a).\]
Since $f(x) = 0$, we take the limit as $\varepsilon \to 0$ in the above inequalities to conclude that $ (f^*)^*(x) = 0 = f(x)$.

The proof is complete.
\end{proof}

\begin{example} \label{ex: dual of Jordan det}
Let $V$ be a finite-dimensional Euclidean Jordan algebra of rank $r$ with unit element $e$, symmetric cone $\P_+$, and the inner product
\[\langle x, y \rangle = \tr\nolimits_{\mathcal{J}}(x \circ y), \quad x, y \in V.\]
Example \ref{ex: Jordan det is e-central hyperbolic} shows that $\det\nolimits_{\mathcal{J}}$ is hyperbolic in $V$ with G\r{a}rding cone $\Gamma_+ = \P_+$. Moreover, under the identification $V^* = V$ via the inner product, the symmetric cone is self-dual, i.e., $\Gamma^*_+ = \Gamma_+ = \P_+$.
Define the function
\[f(x) = \left(\det\nolimits_{\mathcal{J}}x\right)^{\frac{1}{r}} \quad \text{for} \;\; x \in \P.\]
For $a \in \Gamma_+ = \P_+$, \cite[Propositions III.4.2(ii) and III.4.4]{Faraut_book_Analysis_symmetric_cones} give $\nabla \log \det\nolimits_{\mathcal{J}}(a) = a^{-1}$, so
\[\nabla \log f(a) = \frac{1}{r} \nabla \log \det\nolimits_{\mathcal{J}}(a)  = \frac{a^{-1}}{r}.\]
In other words, for every $u \in \Gamma^*_+ =\P_+$, setting $a = (ru)^{-1}$ gives $\nabla \log f(a) = u$. Thus, by Theorem \ref{prop: dual function}(a), we have
\[f^*(u) = \frac{1}{f(a)} = \frac{1}{f((ru)^{-1})} = \frac{1}{\det\nolimits_{\mathcal{J}}((ru)^{-1})^\frac{1}{r}} = \left(\det\nolimits_{\mathcal{J}}(ru)\right)^{\frac{1}{r}} = r \left(\det\nolimits_{\mathcal{J}}u\right)^{\frac{1}{r}}.\]
\end{example}

\begin{example} \label{ex: dual of sigma_2}
This example is taken from \cite[Eq. (1.12)]{Kuo-Trudinger_IUMJ2007_New_max_principle}. We consider the function $\sigma_2$ on $\Real^n$. The dual cone of $\Gamma(\sigma_2)$ is
\[\Gamma^* = \left\{u \in \Real^n : \, \sigma_1(u) \geq \sqrt{(n-1)}|u| \right\}.\]
Put $f(x) = \sigma_2(x)^{\frac{1}{2}}$ for $x \in \Gamma(\sigma_2)$. Then for every $u \in \Gamma^*$, we have
\[f^*(u) = \sqrt{2} \left(\frac{\sigma_1(u)^2}{n-1} - |u|^2\right)^{\frac{1}{2}}.\]
\end{example}

\begin{example} \label{ex: dual of sigma_2 matrices}
This example serves as an analogue of Example \ref{ex: dual of sigma_2} for symmetric matrices. We consider the function $\sigma_2$ on $\SS(n)$. The dual cone of $\Gamma(\sigma_2)$ is
\[\Gamma^* = \left\{A \in \SS(n) : \, \tr(A) \geq \sqrt{n-1}|A| \right\}.\]
Put $f(X) = \sigma_2(X)^{\frac{1}{2}}$ for $X \in \Gamma(\sigma_2)$. Then for every $A \in \Gamma^*$, we have
\[f^*(A) = \sqrt{2} \left(\frac{\tr(A)^2}{n-1} - |A|^2\right)^{\frac{1}{2}}.\]
\end{example}

\begin{remark} \label{rmk: majorization and dual bound for gradient}
Let $p$ and $q$ be hyperbolic polynomials in $V$ of degree $m_p$ and $m_q$, respectively, and suppose that $\Gamma(p)_+ \supset \Gamma(q)_+ \supset \{e\}$. We define
\[f(x) = p(x)^{\frac{1}{m_p}} \;\; \text{for } x \in \Gamma(p) \quad \text{and} \quad g(y) = q(y)^{\frac{1}{m_q}} \;\; \text{for } y \in \Gamma(q).\]
Since $f$ is $1$-homogeneous and concave in $\Gamma(p)_+$, we have
\[f(a) \leq \langle \nabla f(x), a \rangle \quad \text{for all} \;\; a,x \in \Gamma(p)_+,\]
where the equality holds at $x = a$. Therefore, the following equivalences hold
\begin{align*}
\frac{f(a)}{f(e)} \geq \frac{g(a)}{g(e)} \;\; \forall\, a \in \Gamma(q)_+ & \;\; \Longleftrightarrow \;\; \frac{\langle \nabla f(x), a \rangle}{g(a)} \geq \frac{f(e)}{g(e)} \;\; \forall\, x \in \Gamma(p)_+,\, \forall\, a \in \Gamma(q)_+\\
& \;\; \Longleftrightarrow \;\; \inf_{a \in \Gamma(q)_+}  \frac{\langle \nabla f(x), a \rangle}{g(a)} \geq \frac{f(e)}{g(e)} \;\; \forall\, x \in \Gamma(p)_+\\
& \;\; \Longleftrightarrow \;\; g^*(\nabla f(x)) \geq \frac{f(e)}{g(e)} \;\; \forall\, x \in \Gamma(p)_+.
\end{align*}
\end{remark}

Remark \ref{rmk: majorization and dual bound for gradient} and Example \ref{ex: dual of Jordan det} show that Theorem \ref{thm: majorization in Jordan algebra} equivalently gives the following generalization of \cite[Proposition 1.6]{Harvey-Lawson_Duke2025_det_major}.

\begin{corollary}
Let $V$ be a finite-dimensional Euclidean Jordan algebra of rank $r$ with unit element $e$ and symmetric cone $\P_+$. Suppose that $p$ is an $e$-central hyperbolic polynomial of degree $m$ on $V$ with G\r{a}rding cone $\Gamma_+$. We set $f(x) = p(x)^{\frac{1}{m}}$ for $x \in \Gamma_+$. If $\Gamma_+ \supset \P_+$, then
\[\det\nolimits_{\mathcal{J}}(\nabla f(x)) \geq \left(\frac{f(e)}{r}\right)^r \quad \text{for all} \;\; x \in \Gamma_+.\]
\end{corollary}

Similarly, by Remark \ref{rmk: majorization and dual bound for gradient}, Example \ref{ex: dual of sigma_2}, and Corollary \ref{cor: sigma_2 majorization}, we have the following.

\begin{corollary}
Let $p$ be a $\one$-central hyperbolic polynomial of degree $m$ on $\Real^n$ with G\r{a}rding cone $\Gamma(p)_+$. We set $f(x) = p(x)^{\frac{1}{m}}$ for $x \in \Gamma(p)_+$. If $\Gamma(p)_+ \supset \Gamma(\sigma_2)_+$, then
\[\sigma_1(\nabla f(x))^2 - (n-1)|\nabla f(x)|^2 \geq \frac{f(\one)^2}{n} \quad \text{for all} \;\; x \in \Gamma(p)_+.\]
\end{corollary}

In the symmetric matrix setting, Remark \ref{rmk: majorization and dual bound for gradient}, Example \ref{ex: dual of sigma_2 matrices}, and Corollary \ref{cor: sigma_2 majorization in matrices} give the following analogue.

\begin{corollary}
Let $P$ be an $I_n$-central hyperbolic polynomial of degree $m$ on $\SS(n)$ with G\r{a}rding cone $\Gamma(P)_+$. We set $F(X) = P(X)^{\frac{1}{m}}$ for $X \in \Gamma(P)_+$. If $\Gamma(P)_+ \supset \Gamma(\sigma_2)_+$, then
\[\tr(\nabla F(X))^2 - (n-1)|\nabla F(X)|^2 \geq \frac{F(I_n)^2}{n} \quad \text{for all} \;\; X \in \Gamma(P)_+.\]
\end{corollary}

\section*{Acknowledgements}	
The author was supported by National Key R\&D Program of China 2020YFA0712800.

\end{document}